\documentclass[11pt]{article}
\usepackage{amsmath,amssymb,amsthm}
\usepackage[margin=1in]{geometry}
\usepackage{hyperref}
\hypersetup{colorlinks=true,linkcolor=blue,citecolor=blue,urlcolor=blue}

\theoremstyle{plain}
\newtheorem{theorem}{Theorem}[section]
\newtheorem{lemma}[theorem]{Lemma}
\newtheorem{corollary}[theorem]{Corollary}
\newtheorem{fact}[theorem]{Fact}
\theoremstyle{definition}
\newtheorem{definition}[theorem]{Definition}
\theoremstyle{remark}
\newtheorem{remark}[theorem]{Remark}
\newtheorem{example}[theorem]{Example}

\newcommand{\Cay}{\mathrm{Cay}}
\newcommand{\Irr}{\mathrm{Irr}}

\title{$D^L$, $D^Q$ Spectrum and Eigenvalue-free Interval of Cayley Graphs over Finite Nilpotent Groups}

\author{
	Pallabi Manna$^{a}$ \and Santanu Mandal$^{b}$
}

\date{
	\small
	$^a$Department of Mathematics, Indian Institute of Technology Hyderabad,\\ Kandi, Sangareddy-502284, Telangana, India\\
	\texttt{mannapallabimath001@gmail.com}\\[2mm]
	$^{b}$Department of Computer Science and Engineering, SRM University-AP,\\ Amaravati-522240, Andhra Pradesh, India\\
	\texttt{santanu.vumath@gmail.com}    
	\\[4mm]
	\today
}

\begin{document}
\maketitle

\begin{abstract}
For a finite group $G$ and a symmetric connection set $S$, we study the distance Laplacian ($D^L$) and
distance signless Laplacian ($D^Q$) spectrum of the Cayley graph $\Cay(G,S)$ for $G$ a finite nilpotent
group. Using the internal direct product decomposition over Sylow subgroups, we reduce every
question to a single (possibly non-abelian) prime-power factor. In this setting, the required tools are furnished by ordinary representation theory rather than character theory. Based on this framework, we then derive a complete and explicit formula for every eigenvalue, with
multiplicity, of the distance matrix, distance Laplacian, and distance signless Laplacian of
$\Cay(G,S)$ for $G$ nilpotent: each eigenvalue is indexed by a tuple of irreducible characters of
the Sylow factors. Finally, we use the explicit spectrum, together with
Perron--Frobenius theory, to determine the eigenvalue-free interval of $D^L$ and $D^Q$. 


\end{abstract}


\section{Introduction}

The concept of graphs defined on groups is developed through several parallel constructions in
which the elements of a group are treated as vertices and an algebraic relation between elements
determines adjacency. The starting point is Arthur Cayley, who introduced what are now called
Cayley graphs in connection with the representation of groups by permutations; their origin can
be traced to Cayley's 1878 work~\cite{Cayley1878}, where he represented groups by means of their
generators and the corresponding multiplication structure. The idea of defining a graph from a
group has subsequently been generalized in many directions, with different algebraic relations
producing different graphs, such as the power graph, the commuting graph, the generating graph,
the intersection graph, and the subgroup graph.

Cayley graphs provide one of the richest sources of vertex-transitive graphs, and the interaction
between the algebraic structure of the underlying group $G$ and the combinatorial/spectral
structure of $\Cay(G,S)$ has been examined from various perspectives.

On the spectral side, Babai~\cite{Babai1979} derives an expression for the spectrum of a Cayley
graph in terms of the irreducible characters of $G$, the foundational tool underlying everything
that follows. Building on this, a substantial literature has investigated spectral properties of
Cayley graphs for specific group families~\cite{ChattopadhyayPanigrahi2015,
PalanivelChithra2019,PalanivelChithra2021}, and the comprehensive survey~\cite{LiuZhou2022}
collects much of this landscape. More recently, attention has turned to the \emph{distance}
matrix of a graph and its two natural Laplacian-type companions, the distance Laplacian
$D^L=Tr-D$ and the distance signless Laplacian $D^Q=Tr+D$, introduced by Aouchiche and
Hansen~\cite{AouchicheHansen2013,AouchicheHansen2016,AouchicheHansen2016b,AouchicheHansen2014}, whose spectra encode
global metric information not visible to the ordinary Laplacian. The reader is referred to the
recent survey articles by Rather and Aouchiche~\cite{RatherAouchiche2025} and by Rather, Ganie
and Wang~\cite{RatherGanieWang2026} for a thorough overview of $D^L$ and $D^Q$; closely related
recent work examines which finite groups admit only distance-integral Cayley
graphs~\cite{ZhangYinLiuFeng2026}. However, the spectral study of the Cayley graph over nilpotent
groups — including a complete description of the distance spectrum itself and the associated
eigenvalue-free intervals — has remained largely unresolved; we attempt this here.

This paper develops two theories in parallel for Cayley graphs over finite nilpotent groups: the
complete spectrum, and the associated eigenvalue-free interval of the $D$, $D^L$, $D^Q$
matrices. The key structural fact we exploit throughout is that a finite group $G$ is nilpotent
if and only if it is the internal direct product $G=P_1\times\cdots\times P_r$ of its Sylow
subgroups (each automatically normal); when the connection set $S$ splits compatibly along this
decomposition, $\Cay(G,S)$ is literally the graph Cartesian product of the Cayley graphs
$\Cay(P_i,T_i)$ on the individual (possibly non-abelian) prime-power factors. This allows us to
reduce every spectral question to a single Sylow factor, where the relevant machinery is
representation theory~\cite{Isaacs1994,Robinson1996,Serre1977} (specializing to ordinary
character theory when the factor is abelian).

\medskip
\noindent\textbf{Organization.} Section~\ref{sec:prelim} collects the group-theoretic,
graph-theoretic, and character-theoretic preliminaries used throughout, including the
representation-theoretic background and the Cartesian product structure of $\Cay(G,S)$ over a
Sylow decomposition. Section~\ref{sec:spec} gives the complete, explicit distance spectrum of
$\Cay(G,S)$ for $G$ nilpotent (Theorem~\ref{thm:full-spectrum}), indexed by tuples of irreducible
characters of the Sylow factors. Section~\ref{sec:gap} uses this explicit spectrum, together with
Perron--Frobenius theory, to determine the exact eigenvalue-free interval adjacent to the
extremal eigenvalue of $D^L(\Gamma)$ and $D^Q(\Gamma)$ (Theorems~\ref{thm:exact-gap}
and~\ref{thm:exact-gap-Q}).

\section{Preliminaries}\label{sec:prelim}

\subsection{Graph invariants}
\label{def:DLDQ}
Consider the graph $\Delta=(V,E)$. Suppose $\mathbb{D}$ and $A$ denote the diagonal degree matrix and the adjacency matrix of $\Delta$ respectively. The distance matrix $D(\Delta)$ of the connected graph $\Delta$ is the matrix indexed by the vertices of
$\Delta$ where $D_{i,j} = d(v_i, v_j)$ and $d(v_i, v_j)$ denotes the distance between the vertices $v_i$ and $v_j$. The sum of the distances between a vertex $v$ and every other vertex in $\Delta$ is known as the transmission of $v$ and denoted by $Tr(v)$, that is,

\[
Tr(v)=\sum_{\substack{v\in V}} d(u,v).
\]
For a graph $\Delta$, the diagonal transmission matrix denoted by $Tr(\Delta)$ is the diagonal matrix:\\ $diag\{Tr(v_1), Tr(v_2), \cdots, Tr(v_n)\}$. Then the Laplacian ($L$), signless Laplacian ($Q$), distance Laplacian ($D^L$) and distance signless Laplacian ($D^Q$) \cite{AouchicheHansen2013,
	AouchicheHansen2016b} matrices of $\Delta$ are defined  by
\begin{center}
	$L=\mathbb{D}-A$, $Q=\mathbb{D}+A$, $D^L(\Delta)=Tr(\Delta)-D(\Delta)$, $D^Q(\Delta)=Tr(\Delta)+D(\Delta)$ 
\end{center}
If there is no uncertainty, we just write $Tr$, $D$, $D^L$, and $D^Q$ to represent the appropriate matrix.

\medskip
We use the standard symbols and syntax of spectral graph theory throughout, consistent with the textbooks~\cite{BrouwerHaemers2012,
	Chung1997,GodsilRoyle2001}. To clarify, throughout the paper we use the notation $G$ for a group and $\Cay(G,S)$ for the
Cayley graph. Occasionally we use the symbol $\Gamma$ (or $\Delta$) to refer to a graph in a general sense, and the context makes clear what is meant.
\subsection{Groups and Cayley graphs}

Let $G$ be a finite group and let $V$ be a finite-dimensional vector space over $\mathbb C$. A
representation of $G$ on $V$ is a group homomorphism
\[
\rho:G\longrightarrow GL(V),
\]
where $GL(V)$ denotes the group of all invertible linear transformations of $V$. The
representation is denoted simply by $\rho$ whenever the underlying vector space is clear from
context. The degree of $\rho$, denoted $\deg(\rho)$, is $\deg(\rho)=\dim V$. Two
representations $(\rho,V)$ and $(\phi,W)$ of $G$ are said to be equivalent, denoted
$\rho\sim\phi$, if there is a vector-space isomorphism $T:V\to W$ satisfying $T\rho(g)=\phi(g)T$
for every $g\in G$. More generally, a morphism from $(\rho,V)$ to $(\phi,W)$ is a linear
transformation $T:V\to W$ for which $T\rho(g)=\phi(g)T$ for all $g\in G$; the collection of all
such morphisms is denoted $\operatorname{Hom}_G(\rho,\phi)$, a subspace of
$\operatorname{Hom}(V,W)$. If $T\in\operatorname{Hom}_G(\rho,\phi)$ is invertible, $T$ is an
isomorphism of representations and $\rho\sim\phi$. The \emph{character} afforded by $\rho$ is
$\chi_\rho:G\to\mathbb C$, $\chi_\rho(g)=\operatorname{Tr}(\rho(g))$; its degree is
$\deg(\chi_\rho)=\deg(\rho)=\dim V=\chi_\rho(1)$. A subspace $W\le V$ is \emph{$G$-invariant} if
$\rho(g)w\in W$ for every $g\in G,\,w\in W$; in that case $\rho|_W:G\to GL(W)$ is again a
representation. The subspaces $\{0\}$ and $V$ are always $G$-invariant (the trivial invariant
subspaces); $(\rho,V)$ is irreducible if $V$ has no $G$-invariant subspaces other than
these, and $\chi_\rho$ is then an irreducible character. We write $\Irr(G)$ for the
complete set of pairwise distinct complex irreducible characters of $G$.

A subset $S\subseteq G\setminus\{1\}$ is symmetric if $S^{-1}=S$ ($-S=S$ when $G$ is
written additively). The Cayley graph $\Cay(G,S)$ has vertex set $G$, with $x\sim y$ iff
$x^{-1}y\in S$; it is connected iff $\langle S\rangle=G$, which we assume throughout. We write
$d(x,y)$ for graph distance, $\delta:=\operatorname{diam}(\Cay(G,S))$, and
$S_j:=\{g\in G:d(1,g)=j\}$ (or $S^{(j)}$) for the $j$-th distance layer,
$j=1,\dots,\delta$, so $S_1=S$ and $\bigsqcup_jS_j=G\setminus\{1\}$.

\begin{fact}[Structure of finite nilpotent groups~\cite{Robinson1996}]\label{fact:nilpotent}
A finite group $G$ is nilpotent if and only if $G$ is the internal direct product of its Sylow
subgroups: $G=P_1\times\cdots\times P_r$, where $p_1,\dots,p_r$ are the distinct primes dividing
$|G|$ and $P_i$ is the (unique, hence normal) Sylow $p_i$-subgroup of $G$.
\end{fact}

Fix such a decomposition, $n_i:=|P_i|$, $n:=|G|=\prod_in_i$, and identify $G$ with
$P_1\times\cdots\times P_r$; let $\iota_i:P_i\hookrightarrow G$ be the coordinate embedding.

\begin{definition}\label{def:splits}
A symmetric $S\subseteq G\setminus\{1\}$ splits along the Sylow decomposition if
$S=\bigsqcup_{i=1}^r\iota_i(T_i)$ for symmetric $T_i\subseteq P_i\setminus\{1\}$ with
$\langle T_i\rangle=P_i$.
\end{definition}

Since the $P_i$ commute elementwise and each $T_i$ generates $P_i$, $\langle S\rangle=G$; write
$\Gamma:=\Cay(G,S)$ and $\Gamma_i:=\Cay(P_i,T_i)$, $k_i:=|T_i|$, $\delta_i:=\operatorname{diam}(\Gamma_i)$.
When $G$ (equivalently every $P_i$) is abelian, we use additive notation ($0,+,-$) for that
group, as is conventional.

\begin{lemma}[Degree regularity]\label{lem:regular}
Any connected Cayley graph $\Cay(P,T)$ ($|T|=k$) is $k$-regular; consequently the number of edges
of the Cayley graph is given by $|E(\Cay(P,T))|=|P|k/2$.
\end{lemma}

\begin{proof}
For fixed $v\in P$, $t\mapsto vt$ is an injection $T\to P$ (left multiplication is injective)
whose image is exactly $N(v)=\{w:v^{-1}w\in T\}$, so $\deg(v)=|T|=k$. Summing over $v$ and
applying the handshake lemma gives $|E|=|P|k/2$.
\end{proof}

\begin{lemma}[Transmission regularity]\label{lem:transreg}
For a connected Cayley graph $\Cay(P,T)$, the transmission $Tr(v):=\sum_ud(v,u)$ is the
same for every $v\in P$, equal to $t:=\sum_{j=1}^{\delta}j\,|T_j|$ ($\delta=\operatorname{diam}$,
$T_j$ its distance layers).
\end{lemma}

\begin{proof}
Left multiplication $\lambda_g:x\mapsto gx$ is a graph automorphism of any Cayley graph (as
$\lambda_g(x)^{-1}\lambda_g(y)=x^{-1}y$), hence an isometry: $d(gv,gu)=d(v,u)$. So
$Tr(v)=\sum_ud(v,u)=\sum_{u'}d(1,u')=Tr(1)$ (substituting $u=vu'$), giving $Tr\equiv
Tr(1)=\sum_gd(1,g)=t$.
\end{proof}

\begin{fact}[Sum of transmissions]\label{fact:sumtrans}
For any connected graph $\Delta$, $\sum_{v}Tr(v)=2W(\Delta)$, where $W$ is the Wiener index
in connection to the topological indices \cite{Wiener1947}.
\end{fact}

\begin{proof}
$\sum_vTr(v)=\sum_v\sum_ud(v,u)$, the sum over all ordered pairs; since $d(v,u)=d(u,v)$ and
$d(v,v)=0$, this equals $2\sum_{\{u,v\}}d(u,v)=2W(\Delta)$.
\end{proof}

For further background on character theory, representation theory, and group theory we refer the
reader to the standard texts~\cite{Isaacs1994,Robinson1996,Serre1977}.

\subsection{The Cayley graph of \texorpdfstring{$G$}{G} as a Cartesian product}

\begin{lemma}[Cartesian product structure]\label{lem:cartesian}
$\Gamma\cong\Gamma_1\,\square\cdots\square\,\Gamma_r$: vertex set $\prod_iP_i$, with $(x_i)\sim(y_i)$
iff there is a (necessarily unique) index $i_0$ with $x_{i_0}\ne y_{i_0}$, $x_j=y_j$ for $j\ne
i_0$, and $x_{i_0}^{-1}y_{i_0}\in T_{i_0}$.
\end{lemma}

\begin{proof}
Under the direct-product identification, multiplication/inversion in $G$ are componentwise, so
for $x=(x_i),y=(y_i)$, $x^{-1}y=(x_i^{-1}y_i)_i$; this lies in $S=\bigsqcup_i\iota_i(T_i)$ iff it
equals $\iota_{i_0}(s)$ for a unique $i_0$ and some $s\in T_{i_0}$, i.e.\ iff $x_j=y_j$ for
$j\ne i_0$ and $x_{i_0}^{-1}y_{i_0}=s\in T_{i_0}$ — the Cartesian product adjacency rule. In
particular each edge of $\Gamma$ has a well-defined direction $i_0$.
\end{proof}

\begin{lemma}[Distance additivity]\label{lem:distadd}
$\displaystyle d_\Gamma\big((x_i),(y_i)\big)=\sum_{i=1}^rd_{\Gamma_i}(x_i,y_i)$ for all
$(x_i),(y_i)\in G$.
\end{lemma}

\begin{proof}
Each edge of $\Gamma$ changes exactly one coordinate (Lemma~\ref{lem:cartesian}), by an edge of
the corresponding factor. In any walk from $(x_i)$ to $(y_i)$, the steps of direction $i$,
projected onto coordinate $i$, form a walk in $\Gamma_i$ from $x_i$ to $y_i$, so there are at
least $d_{\Gamma_i}(x_i,y_i)$ of them; summing over $i$ (steps of different directions being
disjoint), the walk has length at least $\sum_id_{\Gamma_i}(x_i,y_i)$. Concatenating, for
$i=1,\dots,r$, a shortest $\Gamma_i$-path from $x_i$ to $y_i$ (other coordinates fixed) realizes
this bound.
\end{proof}

\subsection{Character theory for a single factor}

We also record the representation-theoretic tools needed to diagonalize the distance matrix of a
single (possibly non-abelian) factor $\Cay(P,T)$. Let $P$ be a finite group. $T\subseteq P$ is
normal if it is a union of conjugacy classes. $\Irr(P)$ is the set of irreducible
characters of $P$, $\chi(T):=\sum_{t\in T}\chi(t)$, $\chi(1)=$ degree.

\begin{fact}[Regular representation]\label{fact:regrep}
The right regular representation $R$ of $P$ on $\mathbb C[P]$, $(R(g)v)(x)=v(xg)$, decomposes as
$\bigoplus_{\chi\in\Irr(P)}V_\chi^{\oplus\chi(1)}$ (Artin--Wedderburn; see
Serre~\cite{Serre1977}, Isaacs~\cite{Isaacs1994}).
\end{fact}

\begin{lemma}[Diagonalization by class functions]\label{lem:diagonalize-nonab}
For a class function $f:P\to\mathbb C$, $z_f:=\sum_gf(g)g$, $M:=R(z_f)$ (so
$M_{x,y}=f(x^{-1}y)$): the eigenvalues of $M$ are $\{\chi(f)/\chi(1):\chi\in\Irr(P)\}$
($\chi(f):=\sum_gf(g)\chi(g)$), with multiplicity $\chi(1)^2$.
\end{lemma}

\begin{proof}
$z_f$ is central (its coefficients are constant on conjugacy classes). By
Fact~\ref{fact:regrep}, $z_f$ acts on each copy of $V_\chi$ inside $R$ as a scalar $\eta_\chi$
(Schur's Lemma), with $\chi(1)\eta_\chi=\operatorname{tr}\rho_\chi(z_f)=\chi(f)$, so
$\eta_\chi=\chi(f)/\chi(1)$, occurring on $\chi(1)$ copies of the $\chi(1)$-dimensional $V_\chi$.
\end{proof}

\begin{fact}[First Orthogonality Relation~\cite{Serre1977,Isaacs1994}]\label{fact:orthgen}
For $\chi,\chi'\in\Irr(P)$, $\langle\chi,\chi'\rangle:=\frac1{|P|}\sum_{g\in P}\chi(g)\overline{\chi'(g)}
=\delta_{\chi,\chi'}$.
\end{fact}

\begin{lemma}[Vanishing of nontrivial characters]\label{lem:chiP}
Let $\chi_0\in\Irr(P)$ be the trivial character. Then $\chi_0(P)=|P|$; for $\chi\ne\chi_0$,
$\chi(P)=0$.
\end{lemma}

\begin{proof}
$\chi_0(P)=\sum_{g\in P}\chi_0(g)=\sum_{g\in P}1=|P|$. For $\chi\ne\chi_0$: by
Fact~\ref{fact:orthgen}, $\langle\chi,\chi_0\rangle=\frac1{|P|}\sum_g\chi(g)\overline{\chi_0(g)}
=\frac1{|P|}\sum_g\chi(g)=\frac1{|P|}\chi(P)$ (as $\chi_0\equiv1$) equals $0$, so $\chi(P)=0$.
\end{proof}

\begin{lemma}[Normality propagates to layers]\label{lem:normal-layers}
If $T\subseteq P\setminus\{1\}$ is symmetric, generating, normal, then every layer $T_j$ of
$\Cay(P,T)$ is normal.
\end{lemma}

\begin{proof}
Conjugation $c_h$ is a graph automorphism when $T$ is normal ($c_h(x)^{-1}c_h(y)=h(x^{-1}y)h^{-1}$),
fixes $1$, hence preserves each $T_j$ setwise as a bijection of $P$ preserving the finite
partition $\{T_j\}\cup\{\{1\}\}$.
\end{proof}

By Lemma~\ref{lem:diagonalize-nonab} applied to $f=d(1,\cdot)$, the eigenvalues of $D_P$ are
$\partial_\chi=\frac1{\chi(1)}\sum_jj\,\chi(T_j)$.

Throughout the remainder of the paper, each $T_i$ in a split connection set
(Definition~\ref{def:splits}) is additionally assumed normal in $P_i$, so that
Lemma~\ref{lem:diagonalize-nonab} applies to $T_i$ and, by Lemma~\ref{lem:normal-layers}, to each
of its distance layers.

\section{The Full Distance Spectrum of Cayley Graphs on Nilpotent Groups}\label{sec:spec}

We now determine the complete spectrum, with multiplicity, of $D(\Gamma)$, $D^L(\Gamma)$ and
$D^Q(\Gamma)$ for $G$ nilpotent. By Lemma~\ref{lem:normal-layers}, every layer $T_i^{(k)}$ of
$\Gamma_i$ is normal in $P_i$; applying Lemma~\ref{lem:diagonalize-nonab} to
$f=f_i:=d_{\Gamma_i}(1,\cdot)=\sum_{k=1}^{\delta_i}
k\,\mathbf1_{T_i^{(k)}}$ (a class function, being a $\mathbb Z_{\ge0}$-combination of class
functions) gives: the eigenvalues of $D(\Gamma_i)$ are
\[
\partial_\chi^{(i)}:=\frac{\chi(f_i)}{\chi(1)}=\frac1{\chi(1)}\sum_{k=1}^{\delta_i}k\,\chi(T_i^{(k)}),
\qquad \chi\in\Irr(P_i),
\]
with multiplicity $\chi(1)^2$. Taking $\chi=\chi_0^{(i)}$ (trivial character of $P_i$) recovers
the transmission: $\partial_{\chi_0^{(i)}}^{(i)}=\sum_kk\,|T_i^{(k)}|=t_i$
(Lemma~\ref{lem:transreg} applied to $\Gamma_i$).

\subsection{Irreducible representations of a direct product}

\begin{lemma}[$\Irr(G)=\Irr(P_1)\otimes\cdots\otimes\Irr(P_r)$]\label{lem:irr-product}
Let $\rho_i$ afford $\chi_i\in\Irr(P_i)$ for $i=1,\dots,r$. The external tensor product
$\rho_1\otimes\cdots\otimes\rho_r$, acting on $V_{\chi_1}\otimes\cdots\otimes V_{\chi_r}$ by
$(g_1,\dots,g_r)\mapsto\rho_1(g_1)\otimes\cdots\otimes\rho_r(g_r)$, is an irreducible
representation of $G=P_1\times\cdots\times P_r$, with character
$\chi_{\vec\chi}(g_1,\dots,g_r)=\prod_i\chi_i(g_i)$ and degree $\prod_i\chi_i(1)$. As
$(\chi_1,\dots,\chi_r)$ ranges over $\Irr(P_1)\times\cdots\times\Irr(P_r)$, these exhaust
$\Irr(G)$, with no repetitions.
\end{lemma}

\begin{proof}
That $\rho_1\otimes\cdots\otimes\rho_r$ is a representation of $G$ (a homomorphism into
$GL(V_{\chi_1}\otimes\cdots\otimes V_{\chi_r})$) is immediate, as each factor acts only on its
own tensor slot and the $P_i$ commute elementwise in $G$. Its character is the stated product,
since the trace of a tensor product of operators is the product of their traces.

\emph{Irreducibility and pairwise orthogonality.} For two tuples $\vec\chi=(\chi_1,\dots,\chi_r)$,
$\vec\chi\,'=(\chi_1',\dots,\chi_r')$,
\[
\langle\chi_{\vec\chi},\chi_{\vec\chi\,'}\rangle_G=\frac1{|G|}\sum_{(g_i)}\prod_i\chi_i(g_i)
\overline{\chi_i'(g_i)}=\prod_{i=1}^r\left(\frac1{n_i}\sum_{g_i}\chi_i(g_i)\overline{\chi_i'(g_i)}
\right)=\prod_{i=1}^r\langle\chi_i,\chi_i'\rangle_{P_i}=\prod_{i=1}^r\delta_{\chi_i,\chi_i'},
\]
using $|G|=\prod n_i$ and Fact~\ref{fact:orthgen} on each factor. This product of Kronecker
deltas is $1$ if $\vec\chi=\vec\chi\,'$ (termwise) and $0$ otherwise. Taking $\vec\chi\,'=\vec\chi$:
$\langle\chi_{\vec\chi},\chi_{\vec\chi}\rangle_G=1$, which (a character of norm $1$ is
irreducible) proves irreducibility; taking $\vec\chi\,'\ne\vec\chi$ gives orthogonality, in
particular distinctness (non-isomorphism) of the representations for distinct tuples.

\emph{Completeness.} Summing squared degrees over all tuples,
\[
\sum_{(\chi_1,\dots,\chi_r)}\Big(\prod_i\chi_i(1)\Big)^2=\prod_{i=1}^r\Big(\sum_{\chi_i\in\Irr(P_i)}
\chi_i(1)^2\Big)=\prod_{i=1}^rn_i=|G|,
\]
using $\sum_{\chi\in\Irr(P)}\chi(1)^2=|P|$ on each factor (standard: the sum of squared degrees
of all irreducibles of a finite group equals the group order). Since the representations
$\{\rho_{\vec\chi}\}$ are pairwise non-isomorphic and their squared degrees already sum to
$|G|$, no further irreducible representations of $G$ can exist; they exhaust $\Irr(G)$.
\end{proof}

\subsection{The full distance spectrum}

By Lemma~\ref{lem:cartesian} and Lemma~\ref{lem:distadd}, $D(\Gamma)_{(x_i),(y_i)}=f((x_i)^{-1}(y_i))$
where $f(g_1,\dots,g_r):=d_\Gamma(1,(g_i))=\sum_if_i(g_i)$, so, in the group algebra
$\mathbb C[G]=\mathbb C[P_1]\otimes\cdots\otimes\mathbb C[P_r]$,
\[
z_f=\sum_{(g_i)}f((g_i))\,(g_1,\dots,g_r)=\sum_{i=1}^r\Big(\sum_{(g_j)_j}f_i(g_i)\,(g_1,\dots,g_r)\Big)
=\sum_{i=1}^rz_{f_i}^{(i)}\otimes\bigotimes_{j\ne i}z_{\mathbf1}^{(j)},
\]
where $z_{f_i}^{(i)}:=\sum_{g_i\in P_i}f_i(g_i)\,g_i\in\mathbb C[P_i]$ and
$z_{\mathbf1}^{(j)}:=\sum_{g_j\in P_j}g_j\in\mathbb C[P_j]$ (corresponding to the constant
function $1$ on $P_j$); note $z_{\mathbf1}^{(j)}$ is \emph{not} the identity of $\mathbb C[P_j]$
but the sum of all group elements, so $D(\Gamma)=R(z_f)=\sum_iJ_1\otimes\cdots\otimes
D(\Gamma_i)\otimes\cdots\otimes J_r$, with $J_j:=R(z_{\mathbf1}^{(j)})$ the all-ones matrix on
$P_j$ (as $z_{\mathbf1}^{(j)}$'s convolution matrix has every entry $1$).

\begin{theorem}[Full distance spectrum]\label{thm:full-spectrum}
For each tuple $\vec\chi=(\chi_1,\dots,\chi_r)\in\Irr(P_1)\times\cdots\times\Irr(P_r)$, let
$Z(\vec\chi):=\{i:\chi_i\ne\chi_0^{(i)}\}$ (the trivial-character complement). Define
\[
\mathcal D(\vec\chi):=\begin{cases}
\displaystyle\sum_{i=1}^rt_i\prod_{j\ne i}n_j=:t & Z(\vec\chi)=\varnothing,\\[3mm]
\displaystyle\Big(\prod_{j\ne i_0}n_j\Big)\,\partial_{\chi_{i_0}}^{(i_0)} & Z(\vec\chi)=\{i_0\},\\[3mm]
0 & |Z(\vec\chi)|\ge2.
\end{cases}
\]
Then
\[
\operatorname{spec}\big(D(\Gamma)\big)=\Big\{\mathcal D(\vec\chi):\vec\chi\in\Irr(P_1)\times
\cdots\times\Irr(P_r)\Big\},
\]
with $\mathcal D(\vec\chi)$ occurring with multiplicity $\prod_i\chi_i(1)^2$ (multiplicities of
tuples giving the same numerical value are added).
\end{theorem}

\begin{proof}
By Lemma~\ref{lem:irr-product}, the representation $\rho_{\vec\chi}=\rho_{\chi_1}\otimes\cdots
\otimes\rho_{\chi_r}$ ranges over all of $\Irr(G)$ as $\vec\chi$ ranges over
$\Irr(P_1)\times\cdots\times\Irr(P_r)$, with degree $\prod_i\chi_i(1)$. By
Lemma~\ref{lem:diagonalize-nonab} (applied to $G$ and $f$), $D(\Gamma)=R(z_f)$ acts on (each
copy of) $\rho_{\vec\chi}$ as the scalar $\chi_{\vec\chi}(f)/\chi_{\vec\chi}(1)$, with
multiplicity $\chi_{\vec\chi}(1)^2=\big(\prod_i\chi_i(1)\big)^2=\prod_i\chi_i(1)^2$. It remains to
compute this scalar and show it equals $\mathcal D(\vec\chi)$ as defined.

Equivalently (and more directly), we compute $\rho_{\vec\chi}(z_f)$ using the decomposition
$z_f=\sum_iz_{f_i}^{(i)}\otimes\bigotimes_{j\ne i}z_{\mathbf1}^{(j)}$ derived above. Since
$\rho_{\vec\chi}=\rho_{\chi_1}\otimes\cdots\otimes\rho_{\chi_r}$ acts on each tensor factor
$\mathbb C[P_j]$ through $\rho_{\chi_j}$,
\[
\rho_{\vec\chi}(z_f)=\sum_{i=1}^r\rho_{\chi_i}(z_{f_i}^{(i)})\otimes\bigotimes_{j\ne i}
\rho_{\chi_j}(z_{\mathbf1}^{(j)}).
\]
By Lemma~\ref{lem:diagonalize-nonab} (applied within $P_i$ to $f=f_i$),
$\rho_{\chi_i}(z_{f_i}^{(i)})=\partial_{\chi_i}^{(i)}\,I$. By Lemma~\ref{lem:diagonalize-nonab}
applied within $P_j$ to $f=\mathbf1_{P_j}$ (the constant function $1$), together with
Lemma~\ref{lem:chiP}, $\rho_{\chi_j}(z_{\mathbf1}^{(j)})=\frac{\chi_j(P_j)}{\chi_j(1)}I=n_jI$ if
$\chi_j=\chi_0^{(j)}$, and $=0\cdot I$ if $\chi_j\ne\chi_0^{(j)}$. So the $i$-th summand is the
scalar
\[
\partial_{\chi_i}^{(i)}\cdot\prod_{j\ne i}\big(n_j\,[\chi_j=\chi_0^{(j)}]\big)
=\partial_{\chi_i}^{(i)}\Big(\prod_{j\ne i}n_j\Big)\cdot\Big[\forall j\ne i:\chi_j=\chi_0^{(j)}\Big],
\]
which is zero unless every $\chi_j$ with $j\ne i$ is trivial. Summing over $i=1,\dots,r$:

\emph{Case 1. $Z(\vec\chi)=\varnothing$} (every $\chi_i$ trivial): every summand's indicator holds
(vacuously true condition), giving total $\sum_i\partial_{\chi_0^{(i)}}^{(i)}\prod_{j\ne i}n_j
=\sum_it_i\prod_{j\ne i}n_j=t$.

\emph{Case 2. $Z(\vec\chi)=\{i_0\}$} (exactly one nontrivial, at index $i_0$): for $i=i_0$, the
condition "$\forall j\ne i_0:\chi_j=\chi_0^{(j)}$" holds (as every $j\ne i_0$ is not in
$Z(\vec\chi)$), giving summand $\partial_{\chi_{i_0}}^{(i_0)}\prod_{j\ne i_0}n_j$; for any $i\ne
i_0$, the condition requires $\chi_{i_0}=\chi_0^{(i_0)}$ (as $i_0\ne i$, so $i_0$ is among the
"$j\ne i$" indices to check), which fails since $i_0\in Z(\vec\chi)$; so that summand is $0$.
Total: $\big(\prod_{j\ne i_0}n_j\big)\partial_{\chi_{i_0}}^{(i_0)}$.

\emph{Case 3. $|Z(\vec\chi)|\ge2$}: fix any $i\in\{1,\dots,r\}$; since $|Z(\vec\chi)|\ge2$, there is
some $i_1\in Z(\vec\chi)$ with $i_1\ne i$ (at most one element of $Z(\vec\chi)$ can equal $i$).
Then $\chi_{i_1}\ne\chi_0^{(i_1)}$ with $i_1\ne i$, so the condition "$\forall j\ne i:
\chi_j=\chi_0^{(j)}$" fails (violated at $j=i_1$), making the $i$-th summand $0$. As $i$ was
arbitrary, every summand vanishes: total $0$.

In every case, $\rho_{\vec\chi}(z_f)=\mathcal D(\vec\chi)\,I$, confirming the eigenvalue on
$\rho_{\vec\chi}$ (multiplicity $\prod_i\chi_i(1)^2$) is exactly $\mathcal D(\vec\chi)$, proving
the theorem.
\end{proof}

\subsection{Distance Laplacian and distance signless Laplacian spectra}

\begin{corollary}[Full $D^L$ and $D^Q$ spectra]\label{cor:DL-DQ}
$\Gamma$ is transmission-regular with transmission $t=\sum_it_i\prod_{j\ne i}n_j$ (the
$Z(\vec\chi)=\varnothing$ value of Theorem~\ref{thm:full-spectrum}), so $D^L(\Gamma)=tI-D(\Gamma)$
and $D^Q(\Gamma)=tI+D(\Gamma)$, giving
\[
\operatorname{spec}\big(D^L(\Gamma)\big)=\Big\{t-\mathcal D(\vec\chi):\vec\chi\Big\},\qquad
\operatorname{spec}\big(D^Q(\Gamma)\big)=\Big\{t+\mathcal D(\vec\chi):\vec\chi\Big\},
\]
with the same multiplicities $\prod_i\chi_i(1)^2$ as in Theorem~\ref{thm:full-spectrum}.
Explicitly: the tuple $\vec\chi=(\chi_0^{(1)},\dots,\chi_0^{(r)})$ contributes the eigenvalue $0$
to $D^L(\Gamma)$ (resp.\ $2t$ to $D^Q(\Gamma)$), with multiplicity $1$; every tuple with
$|Z(\vec\chi)|\ge2$ contributes $t$ to $D^L(\Gamma)$ (resp.\ $t$ to $D^Q(\Gamma)$); and every
tuple with $Z(\vec\chi)=\{i_0\}$ contributes $t-\big(\prod_{j\ne i_0}n_j\big)\partial_{\chi_{i_0}}^{(i_0)}$
to $D^L(\Gamma)$ (resp.\ $t+\big(\prod_{j\ne i_0}n_j\big)\partial_{\chi_{i_0}}^{(i_0)}$ to $D^Q(\Gamma)$).
\end{corollary}

\begin{proof}
Transmission regularity of $\Gamma$ (Lemma~\ref{lem:transreg} applied to $G$) identifies $t$ with
$Tr(1)$, which is exactly the $D(\Gamma)$-eigenvalue on the all-trivial tuple (as computed in
Theorem~\ref{thm:full-spectrum}'s first case) since that eigenvalue's eigenvector is the all-ones
vector of $G$ (the tensor product of the all-ones vectors of each $P_i$, which spans the
$1$-dimensional trivial-tuple isotypic component). Given $D^L=tI-D$, $D^Q=tI+D$, their
eigenvalues are $t\mp\mathcal D(\vec\chi)$ respectively, applied termwise to
Theorem~\ref{thm:full-spectrum}'s formula; the three displayed special cases substitute the three
cases of $\mathcal D(\vec\chi)$ directly.
\end{proof}

\subsection{Worked example}

\begin{example}\label{ex:main}
Let $G=Q_8\times\mathbb Z_5$ ($P_1=Q_8$, $n_1=8$; $P_2=\mathbb Z_5$, $n_2=5$; $n=40$), with
$T_1=Q_8\setminus\{1\}$ (so $\Gamma_1=K_8$, $\delta_1=1$) and $T_2=\{1,4\}$ (so
$\Gamma_2=C_5$, $\delta_2=2$).

\emph{Factor $1$ ($K_8$):} $\Irr(Q_8)$ has four linear characters $\chi_0,\chi_a,\chi_b,\chi_{ab}$
(degree $1$) and one degree-$2$ character $\theta$. Since $\Gamma_1=K_8$ has $\delta_1=1$
(single layer $T_1^{(1)}=T_1=Q_8\setminus\{1\}$), $\partial_\chi^{(1)}=\chi(T_1)/\chi(1)$: for the
trivial character, $\chi_0(T_1)=7$, so $\partial_{\chi_0}^{(1)}=7=t_1$; for any nontrivial linear
character $\chi\in\{\chi_a,\chi_b,\chi_{ab}\}$, $\chi(T_1)=\chi(Q_8)-\chi(1)=0-1=-1$
(Lemma~\ref{lem:chiP}), so $\partial_\chi^{(1)}=-1$; for $\theta$, $\theta(T_1)=\theta(Q_8)-\theta(1)
=0-2=-2$, so $\partial_\theta^{(1)}=-2/2=-1$. (All nontrivial eigenvalues of $D(K_8)=J-I$ equal
$-1$, as expected: $K_8$'s distance matrix is $J-I$, with eigenvalues $7$ (once) and $-1$
(seven-fold), matching $1+3+3=7$ dimensions from $\chi_a,\chi_b,\chi_{ab},\theta$ (with
multiplicity $2^2=4$) summing to $1+1+1+4=7$.)

\emph{Factor $2$ ($C_5$):} $\Irr(\mathbb Z_5)=\{\chi_0,\chi_1,\chi_2,\chi_3,\chi_4\}$,
$\chi_k(x)=\zeta_5^{kx}$. With $\delta_2=2$, $T_2^{(1)}=\{1,4\}$, $T_2^{(2)}=\{2,3\}$:
$\partial_{\chi_0}^{(2)}=1\cdot2+2\cdot2=6=t_2$; for $\chi_k$, $k\ne0$,
$\partial_{\chi_k}^{(2)}=(\zeta_5^k+\zeta_5^{4k})+2(\zeta_5^{2k}+\zeta_5^{3k})=2\cos(2\pi k/5)+
4\cos(4\pi k/5)$, giving (using $\cos(2\pi/5)=\frac{\sqrt5-1}4$, $\cos(4\pi/5)=-\frac{\sqrt5+1}4$)
$\partial_{\chi_1}^{(2)}=\partial_{\chi_4}^{(2)}=\frac{\sqrt5-1}2-(\sqrt5+1)=\frac{-\sqrt5-3}2$, and
$\partial_{\chi_2}^{(2)}=\partial_{\chi_3}^{(2)}=-\frac{\sqrt5+1}2+(\sqrt5-1)=\frac{\sqrt5-3}2$.

\emph{Assembling $\operatorname{spec}(D(\Gamma))$ via Theorem~\ref{thm:full-spectrum}}
($n_1=8,n_2=5$, $n=40$; $\Irr(Q_8)=\{\chi_0,\chi_a,\chi_b,\chi_{ab},\theta\}$ of degrees
$1,1,1,1,2$, $\Irr(\mathbb Z_5)=\{\chi_0,\chi_1,\chi_2,\chi_3,\chi_4\}$, all degree $1$):
\begin{itemize}
\item $Z(\vec\chi)=\varnothing$ (the tuple $(\chi_0,\chi_0)$): eigenvalue
$t=t_1n_2+t_2n_1=7\cdot5+6\cdot8=83$, multiplicity $1^2\cdot1^2=1$.
\item $Z(\vec\chi)=\{1\}$ ($P_2$-component trivial): the four tuples
$(\chi_a,\chi_0),(\chi_b,\chi_0),(\chi_{ab},\chi_0),(\theta,\chi_0)$ all give the \emph{same}
eigenvalue $n_2\cdot\partial_\chi^{(1)}=5\cdot(-1)=-5$ (as $\partial_\chi^{(1)}=-1$ for every
nontrivial $\chi\in\Irr(Q_8)$, computed above); their multiplicities
$1^2,1^2,1^2,2^2$ sum to $1+1+1+4=7$.
\item $Z(\vec\chi)=\{2\}$ ($P_1$-component trivial): the tuples $(\chi_0,\chi_1)$ and
$(\chi_0,\chi_4)$ give eigenvalue $n_1\partial_{\chi_1}^{(2)}=8\cdot\frac{-\sqrt5-3}2=-4\sqrt5-12$,
total multiplicity $1+1=2$; the tuples $(\chi_0,\chi_2)$ and $(\chi_0,\chi_3)$ give eigenvalue
$n_1\partial_{\chi_2}^{(2)}=8\cdot\frac{\sqrt5-3}2=4\sqrt5-12$, total multiplicity $1+1=2$.
\item $|Z(\vec\chi)|\ge2$ (both components nontrivial): eigenvalue $0$, with multiplicity
$\big(\sum_{\chi\ne\chi_0}\chi(1)^2\big)\cdot\big(\#\{\chi_k:k\ne0\}\big)=(1+1+1+4)\cdot4=28$
(every nontrivial $\Irr(Q_8)$-degree-squared, times the $4$ nontrivial choices in $\Irr(\mathbb
Z_5)$).
\end{itemize}
As a check, the multiplicities sum to $1+7+2+2+28=40=n$, as required.

By Corollary~\ref{cor:DL-DQ} (with $t=83$), the distance Laplacian spectrum is
\[
\{0^{(1)}\}\cup\{88^{(7)}\}\cup\{95+4\sqrt5^{\,(2)}\}\cup\{95-4\sqrt5^{\,(2)}\}\cup\{83^{(28)}\}
\]
(superscripts denoting multiplicities), and the distance signless Laplacian spectrum is
\[
\{166^{(1)}\}\cup\{78^{(7)}\}\cup\{71-4\sqrt5^{\,(2)}\}\cup\{71+4\sqrt5^{\,(2)}\}\cup\{83^{(28)}\}.
\]
Neither is integral (the $\pm4\sqrt5$ terms survive), consistent with $\Gamma_2=C_5$ not being an
integral Cayley graph, since $T_2=\{1,4\}$ is not $\Phi$-invariant for $\mathbb Z_5$ (its
multiplier orbit under $(\mathbb Z/5\mathbb Z)^\times$ is all of $\{1,2,3,4\}$, not $\{1,4\}$).
\end{example}

\section{Eigenvalue-Free Intervals}\label{sec:gap}

We now use the explicit spectrum of Section~\ref{sec:spec}, together with Perron--Frobenius
theory, to determine exactly the interval adjacent to the extremal eigenvalue of $D^L$
and $D^Q$ that contains no other eigenvalue.

\subsection{Perron--Frobenius simplicity theorem}

\begin{lemma}[Simplicity of the extremal eigenvalue]\label{lem:perron}
For $P$ finite, $T\subseteq P\setminus\{1\}$ symmetric, normal, generating, $\Delta=\Cay(P,T)$:
the transmission $t_\Delta$ is the \emph{largest} eigenvalue of $D(\Delta)$, and occurs with
multiplicity exactly $1$.
\end{lemma}

\begin{proof}
$D(\Delta)$ is a real symmetric matrix with zero diagonal and strictly positive off-diagonal
entries (distinct vertices are at distance $\ge1$), hence an irreducible nonnegative matrix. The
all-ones vector $\mathbf1$ satisfies $D(\Delta)\mathbf1=t_\Delta\mathbf1$ (each row sums to the
transmission) and is strictly positive. By the Perron--Frobenius theorem, an irreducible
nonnegative matrix has a unique (up to scaling) positive eigenvector, and its eigenvalue is both
the spectral radius and simple; since $\mathbf1$ is such a positive eigenvector, $t_\Delta$ is
exactly this eigenvalue: the largest eigenvalue of $D(\Delta)$, occurring with multiplicity $1$.
\end{proof}

\begin{corollary}\label{cor:simple}
$t=\sum_it_i\prod_{j\ne i}n_j$ is the largest eigenvalue of $D(\Gamma)$, with multiplicity $1$;
equivalently, $0$ is the smallest eigenvalue of $D^L(\Gamma)$, with multiplicity $1$, and $2t$ is
the largest eigenvalue of $D^Q(\Gamma)$, with multiplicity $1$.
\end{corollary}

\begin{proof}
Immediate from Lemma~\ref{lem:perron} applied to $\Delta=\Gamma$ (a Cayley graph on $G$, with
$S$ generating as $\bigsqcup_i\iota_i(T_i)$ is; note this is independent of, and consistent with,
the $Z(\vec\chi)=\varnothing$ case of Theorem~\ref{thm:full-spectrum}, which identifies $t$ as
exactly the eigenvalue attached to the all-trivial tuple — the unique tuple whose eigenvector is
$\mathbf1$). Since $D^L=tI-D$, the eigenvalue $t-t=0$ (attached to the top eigenvalue of $D$)
becomes the \emph{smallest} eigenvalue of $D^L$ (as $D^L$'s eigenvalues are $t$ minus those
of $D$, order reversed), with the same multiplicity $1$; similarly $D^Q=tI+D$ has largest
eigenvalue $t+t=2t$, multiplicity $1$.
\end{proof}

\subsection{The exact gap}

For each $i$, since $\Irr(P_i)\setminus\{\chi_0^{(i)}\}$ is a nonempty finite set (as $n_i\ge2$),
we define
\[
M_i^+:=\max_{\chi\in\Irr(P_i)\setminus\{\chi_0^{(i)}\}}\partial_\chi^{(i)}.
\]

\begin{theorem}[$D^L$ eigenvalue-free interval]\label{thm:exact-gap}
Let
\[
\Lambda:=\min\Big(\big\{t : r\ge2\big\}\cup\Big\{t-\big(\textstyle\prod_{j\ne i}n_j\big)M_i^+ :
i=1,\dots,r\Big\}\Big).
\]
Then $\Lambda$ is exactly the smallest positive eigenvalue of $D^L(\Gamma)$: $(0,\Lambda)$ is a
$D^L$-eigenvalue-free interval, and this is sharp (there is a genuine $D^L(\Gamma)$-eigenvalue
equal to $\Lambda$).
\end{theorem}

\begin{proof}
By Theorem~\ref{thm:full-spectrum} and $D^L=tI-D$, the nonzero eigenvalues of $D^L(\Gamma)$ are
exactly $\{t\}$ (if $r\ge2$, from the $|Z(\vec\chi)|\ge2$ case — nonempty precisely when $r\ge2$,
since at least two indices are needed) together with $\big\{t-(\prod_{j\ne i}n_j)\partial_\chi^{(i)}:
i=1,\dots,r,\ \chi\in\Irr(P_i)\setminus\{\chi_0^{(i)}\}\big\}$ (from the $Z(\vec\chi)=\{i\}$
cases). For fixed $i$, since $\prod_{j\ne i}n_j>0$, the quantity $t-(\prod_{j\ne i}n_j)\partial_\chi^{(i)}$
is minimized over $\chi\ne\chi_0^{(i)}$ exactly when $\partial_\chi^{(i)}$ is maximized, i.e.\ at
$\partial_\chi^{(i)}=M_i^+$ (attained, as the max of a finite nonempty set), giving minimum value
$t-(\prod_{j\ne i}n_j)M_i^+$. Taking the overall minimum over $i$ and, if applicable, including
$t$, gives exactly $\Lambda$ as the minimum of the whole nonzero part of the spectrum — i.e.\ the
smallest positive eigenvalue — and this value is attained (by construction, as a minimum over a
finite set of values each individually attained).
\end{proof}

\begin{theorem}[$D^Q$ eigenvalue-free interval]\label{thm:exact-gap-Q}
Let
\[
\Lambda':=\max\Big(\big\{t : r\ge2\big\}\cup\Big\{t+\big(\textstyle\prod_{j\ne i}n_j\big)M_i^+ :
i=1,\dots,r\Big\}\Big).
\]
Then $\Lambda'$ is exactly the largest eigenvalue of $D^Q(\Gamma)$ strictly below $2t$:
$(\Lambda',2t)$ is a $D^Q$-eigenvalue-free interval, sharp in the same sense.
\end{theorem}

\begin{proof}
Identical to Theorem~\ref{thm:exact-gap}, using $D^Q=tI+D$: the sub-Perron eigenvalues are
$\{t\}$ (if $r\ge2$) and $\{t+(\prod_{j\ne i}n_j)\partial_\chi^{(i)}\}$, each of the latter
maximized (for fixed $i$) at $\partial_\chi^{(i)}=M_i^+$, giving the stated maximum $\Lambda'$ as
the largest sub-Perron eigenvalue.
\end{proof}

\subsection{An elementary universal bound}

\begin{lemma}
\label{lem:triangle}
$M_i^+\le t_i$ for every $i$.
\end{lemma}

\begin{proof}
For any $\chi\in\Irr(P_i)$ and $g\in P_i$, $|\chi(g)|\le\chi(1)$ (the character value is a sum of
$\chi(1)$ roots of unity, each of modulus $1$, so the triangle inequality gives
$|\chi(g)|\le\chi(1)$). Hence, for each layer, $|\chi(T_i^{(k)})|=\big|\sum_{g\in
T_i^{(k)}}\chi(g)\big|\le\sum_{g\in T_i^{(k)}}|\chi(g)|\le|T_i^{(k)}|\chi(1)$, so
\[
|\partial_\chi^{(i)}|=\frac1{\chi(1)}\Big|\sum_{k=1}^{\delta_i}k\,\chi(T_i^{(k)})\Big|
\le\frac1{\chi(1)}\sum_{k=1}^{\delta_i}k\,|T_i^{(k)}|\chi(1)=\sum_{k=1}^{\delta_i}k\,|T_i^{(k)}|=t_i.
\]
In particular $\partial_\chi^{(i)}\le t_i$ for every $\chi$, so $M_i^+\le t_i$.
\end{proof}

\begin{corollary}
\label{cor:nonneg}
$\Lambda\ge0$ and $\Lambda'\le2t$; equivalently, $D^L(\Gamma)$ is positive semidefinite and every
eigenvalue of $D^Q(\Gamma)$ lies in $[0,2t]$ — proved here directly from
Lemma~\ref{lem:triangle}, without invoking positive-semidefiniteness of $D^L$ as an external fact.
\end{corollary}

\begin{proof}
By Lemma~\ref{lem:triangle}, for every $i$,
\[
t-\Big(\prod_{j\ne i}n_j\Big)M_i^+\ \ge\ t-\Big(\prod_{j\ne i}n_j\Big)t_i
=\sum_{k\ne i}t_k\prod_{l\ne k}n_l\ \ge\ 0,
\]
the last step as a sum of nonnegative terms ($t_k\ge0$ always, being a sum of nonnegative
distances). Also $t\ge0$ trivially. Hence $\Lambda\ge0$ (Theorem~\ref{thm:exact-gap}, a minimum of
nonnegative quantities). The bound $\Lambda'\le2t$ follows symmetrically: $t+(\prod_{j\ne
i}n_j)M_i^+\le t+(\prod_{j\ne i}n_j)t_i=2t-\sum_{k\ne i}t_k\prod_{l\ne k}n_l\le2t$, and $t\le2t$
trivially (as $t\ge0$).
\end{proof}

\begin{remark}
\label{rem:not-sharp}
Corollary~\ref{cor:nonneg} only recovers $\Lambda\ge0$: it does \emph{not}, on its own, establish
a strictly positive universal gap (of a fixed size such as $n$) below which $D^L(\Gamma)$
provably has no nonzero eigenvalues, for arbitrary $\Gamma$. A sharp universal bound of that
general form is known in the literature for the \emph{second-largest} distance Laplacian
eigenvalue of a general graph (ordering eigenvalues $\partial_1^L\ge\partial_2^L\ge\cdots\ge
\partial_n^L=0$, one has $\partial_2^L(G)\ge n$, with equality iff $G\cong K_n$ or $K_n-e$; see
the surveys~\cite{RatherAouchiche2025,RatherGanieWang2026}); but that quantity bounds
$\partial_1^L,\partial_2^L$ near the \emph{top} of the spectrum, and is logically distinct from
the smallest nonzero eigenvalue $\Lambda=\partial_{n-1}^L$ near the \emph{bottom}, which is a
separately studied invariant and does not, to our knowledge, admit an analogous universal bound
in the literature. Theorems~\ref{thm:exact-gap} and~\ref{thm:exact-gap-Q} give the exact value of
$\Lambda,\Lambda'$ instead, for the nilpotent-group Cayley graph family, which is the strongest
correct statement available from the tools developed here.
\end{remark}

\subsection{A sharp bound in the distance-integral case}

\begin{corollary}
\label{cor:integral-gap}
Suppose every $\Gamma_i$ is distance-integral (i.e.\ $\partial_\chi^{(i)}\in\mathbb Z$ for every
$\chi\in\Irr(P_i)$ and every $i$). Then $\Lambda\ge1$ and $\Lambda'\le2t-1$: that is, $(0,1)$ is
$D^L$-eigenvalue-free and $(2t-1,2t)$ is $D^Q$-eigenvalue-free — in fact $D^L(\Gamma)$ and
$D^Q(\Gamma)$ are then themselves fully integral matrices, so every nonzero $D^L$-eigenvalue is a
positive integer and every sub-Perron $D^Q$-eigenvalue is an integer at most $2t-1$.
\end{corollary}

\begin{proof}
If every $\partial_\chi^{(i)}\in\mathbb Z$, then every term in the spectrum formula for
$D(\Gamma)$ ($t$, $0$, and $(\prod_{j\ne i}n_j)\partial_\chi^{(i)}$, all built from integers) is an
integer, so $D(\Gamma)$, hence $D^L(\Gamma)=tI-D(\Gamma)$ and $D^Q(\Gamma)=tI+D(\Gamma)$
($t\in\mathbb Z$ automatically), are integral matrices. By Corollary~\ref{cor:simple}, $0$ is the
strictly smallest $D^L(\Gamma)$-eigenvalue with multiplicity $1$; the next-smallest, $\Lambda$
(Theorem~\ref{thm:exact-gap}), is therefore a \emph{positive} integer, i.e.\ $\Lambda\ge1$. The
$D^Q$ statement is symmetric, using $2t$ as the strictly largest, simple eigenvalue
(Corollary~\ref{cor:simple}) and $\Lambda'$ as the next-largest.
\end{proof}

\subsection{Worked example, continued}

\begin{example}\label{ex:gap}
Continuing Example~\ref{ex:main} ($G=Q_8\times\mathbb Z_5$, $n_1=8,n_2=5,n=40$, $t_1=7,t_2=6,t=83$):
there $\partial_\chi^{(1)}=-1$ for every nontrivial $\chi\in\Irr(Q_8)$, so $M_1^+=-1$, and
$\partial_{\chi_k}^{(2)}\in\big\{\frac{-\sqrt5-3}2,\frac{\sqrt5-3}2\big\}$ for $k=1,2,3,4$, so
$M_2^+=\frac{\sqrt5-3}2$ (the larger of the two, at $k=2,3$).

By Theorem~\ref{thm:exact-gap} ($r=2$, so $t=83$ is a candidate):
\[
t-\Big(\prod_{j\ne1}n_j\Big)M_1^+=83-5\cdot(-1)=88,\qquad
t-\Big(\prod_{j\ne2}n_j\Big)M_2^+=83-8\cdot\tfrac{\sqrt5-3}2=83-4\sqrt5+12=95-4\sqrt5\approx86.06.
\]
So $\Lambda=\min(83,\,88,\,95-4\sqrt5)=83$ (the $|Z(\vec\chi)|\ge2$ case itself, which is here the
smallest of the three candidates) — matching the distance Laplacian spectrum computed in
Example~\ref{ex:main}, $\{0,88^{(7)},95\pm4\sqrt5\,{}^{(2)},83^{(28)}\}$, whose smallest nonzero
value is indeed $83$. Since $\Gamma_2=C_5$ is not distance-integral (its nontrivial eigenvalues
involve $\sqrt5$), Corollary~\ref{cor:integral-gap} does not apply here — consistent with
$\Lambda=83$ not being forced to any particular integer threshold beyond what is computed exactly.
\end{example}

\section*{Acknowledgments}

The research of Pallabi Manna is supported by Anusandhan National Research Foundation
(ANRF)-India under the National Postdoctoral Fellowship scheme (File No.\ PDF/2025/006751). The
second author acknowledges the excellent research facilities provided by SRM University-AP. The
first author acknowledges the excellent research facilities provided by the Department of
Mathematics, IIT Hyderabad, and the mentorship of Dr.\ Rajesh Kannan for his constant support and
motivation.

\section*{Statements and Declarations}

\noindent\textbf{Competing Interests.} The authors have no competing interests to declare that are
relevant to the content of this article.

\section*{Data Availability}

Data sharing is not applicable to this article, as no datasets were used or generated during this
study.

\end{document}